\documentclass[a4paper]{amsart}
\usepackage{color}
\usepackage[english]{babel}
\usepackage[utf8]{inputenc}
\usepackage{amssymb}
\usepackage{amsmath}
\usepackage{amsthm}
\usepackage[justification=centering]{caption}
\usepackage{dsfont}
\usepackage{enumerate}
\usepackage{graphicx}
\usepackage{hyperref}
\usepackage{makecell}
\usepackage{mathrsfs}
\usepackage{multirow}
\usepackage[table]{xcolor}

\newtheorem{Theorem}{Theorem}[section]
\newtheorem{Conjecture}[Theorem]{Conjecture}
\newtheorem{Lemma}[Theorem]{Lemma}
\newtheorem{Property}[Theorem]{Property}
\newtheorem{Proposition}[Theorem]{Proposition}

\theoremstyle{definition}

\theoremstyle{remark}
\newtheorem{Remark}[Theorem]{Remark}

\numberwithin{equation}{section}
\numberwithin{equation}{subsection}

\usepackage{mathtools}

\newcommand{\ceps}{\reflectbox{$3$}}

\begin{document}

	\title{Elementary slopes of plane-wide breakout}
	\author{Yann Jullian}
	\address[Yann Jullian]{Ingénieur CaSciModOT, Institut Denis Poisson, CNRS UMR 7013, Université de Tours}
	\email{yann.jullian@univ-tours.fr}

	\begin{abstract}
		We consider a mathematical model of the \textit{breakout} game where $\mathds{R}^2$
		is covered by unit square \textit{bricks} everywhere except in one place.
		We introduce \textit{elementary blocks} (particular sets of bricks) in order to
		describe a family of \textit{periodic} (in a \textit{breakout}-appropriate sense)
		orbits. Necessary and sufficient conditions for slopes (the initial direction of the ball)
		to produce these orbits are given.
	\end{abstract}

	\maketitle

	\setcounter{tocdepth}{3}
	\tableofcontents
	\newpage

	\section{Introduction}
		This article deals with a mathematical model of the popular game \textit{breakout}. In the game,
		a set of bricks is placed in a space delimited by walls and a ball is launched. The ball bounces
		similarly off the walls and the bricks, but when a brick is hit, it is immediately destroyed,
		opening up the space in which the ball can travel. The player controls a pallet at the bottom
		of the screen and must prevent the ball from escaping by having it bounce, as it would off a wall, off the pallet.
		The aim is to destroy all the bricks. The mathematical model forgoes the player entirely, but keeps
		the main principle of bricks being destroyed when hit.\\

		The article \cite{BF} should be considered the reference when dealing with the mathematics of breakout. As such, it sets
		definitions and notations in the most general of cases. Following from \cite{BF}, a general setup
		specifies a (polygonal, not necessarily bounded) subset of $\mathds{R}^2$ as a domain
		where obstacles (compacts polygons) are placed. A ball is launched from an unobstructed
		(not occupied by an obstacle) point of the domain and travels at constant speed. Upon hitting
		the frontier of the domain, the ball bounces as it would in a billiard (with the incidence angle being
		equal to the reflection angle). When the ball hits an obstacle, the ball bounces similarly,
		but the obstacle is destroyed in the process.\\

		We consider only a single case in the present article.
		The domain is $\mathds{R}^2$ (there are no frontiers) and the obstacles are the unit squares
		$\mathcal{O} = \{[x, x+1]\times [y, y+1], (x, y)\in \mathds{Z}^2\setminus \{(-1, 0)\}\}$.
		Obstacles will simply be called \textbf{bricks}, and all are present initially.

		The general question that this article attempts to partially answer is whether the behavior observed
		in square billiards has an equivalent in such a system. Namely, a rational slope (the initial direction of the ball)
		in square billiards will always result in a periodic orbit. There can be no periodic orbits in our setup
		since the number of bricks keeps decreasing, but we can look for periodicity in the sequence of bricks being destroyed.
		In \cite{BF}, \textit{relative periodicity} is introduced as a canonical way to express periodicity in a breakout setup.
		Loosely, an orbit is said to be relatively periodic if the ball repeatedly breaks the same, up to translations, blocks of bricks.
		An illustration of a relatively periodic orbit is given on figure \ref{fig:01_02} while the proper definition is recalled in section \ref{sec:definitions}.
		\begin{Remark}
			All of our figures follow the same convention. Each repeating block is given its own color and for each
			block, the bricks are numbered, starting from $0$, in the order they would be destroyed by the ball.
			The black square represents brick $0$ of the next block.
		\end{Remark}
		\begin{figure}[h!]
			\begin{center}
				\scalebox{0.5}{\includegraphics{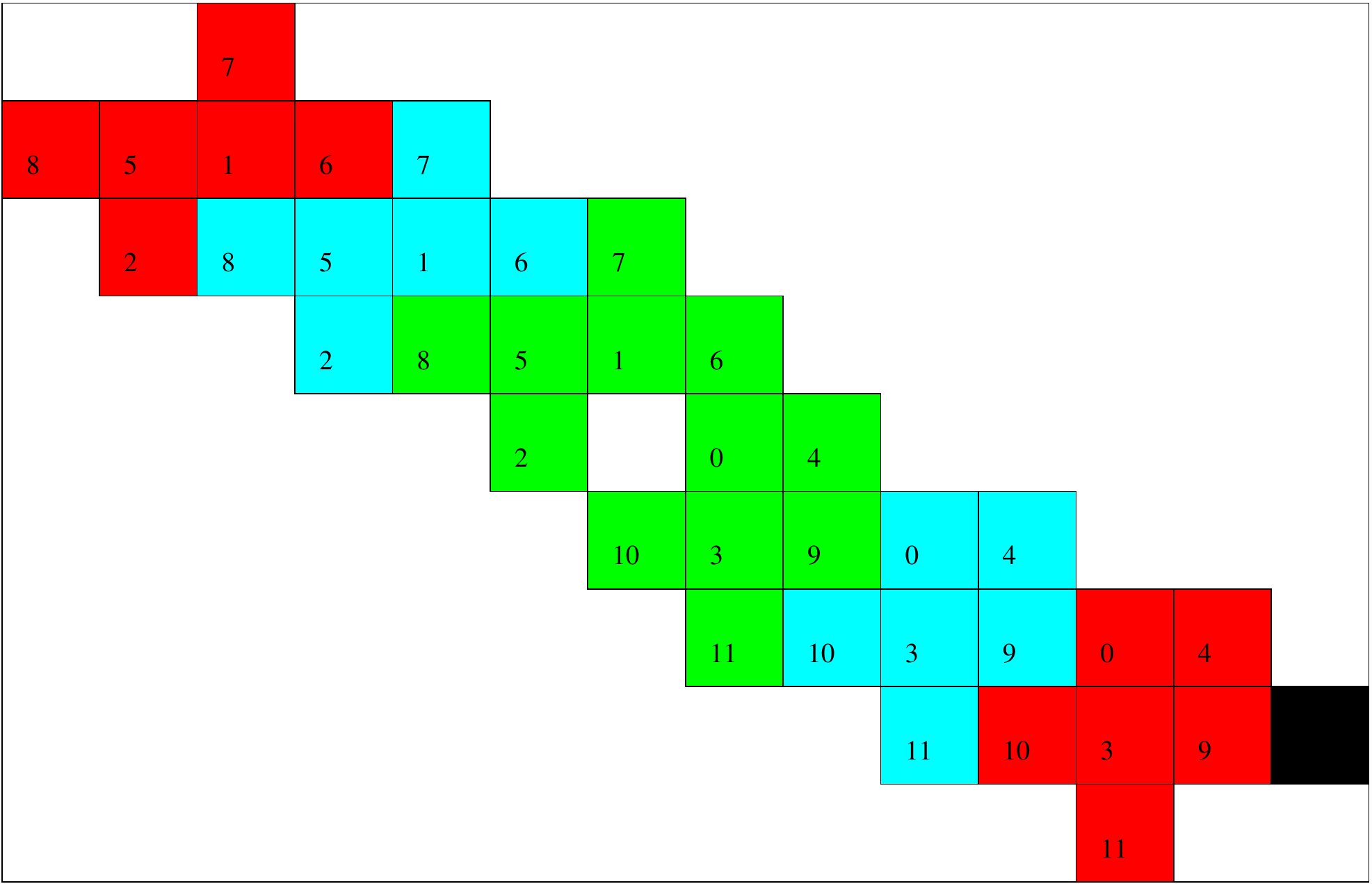}}
			\end{center}
			\vspace{-4mm}
			\caption{Relative periodicity of slope $\frac{1}{2}$.}
			\label{fig:01_02}
		\end{figure}

		The general question then becomes:
		with $\mathds{R}^2$ as the domain and $\mathcal{O} = \{[x, x+1]\times [y, y+1], (x, y)\in \mathds{Z}^2\setminus \{(-1, 0)\}\}$
		as the set of obstacles, do rational slopes produce relatively periodic orbits?
		Only a few of such slopes are given in \cite{BF}.\\

		The present article exposes a family of slopes with relatively periodic orbits.
		We introduce \textit{elementary blocks} as finite sets of bricks composed of
		\begin{itemize}
			\item a single \textit{top} brick with (lower left corner's) ordinate $Y+1$ for some $Y\in \mathds{Z}$,
			\item a single \textit{bottom} brick with (lower left corner's) ordinate $Y-1$,
			\item a one-brick-missing contiguous set of bricks with (lower left corner's) ordinate $Y$.
		\end{itemize}
		Two such elementary blocks are depicted in figure \ref{fig:03_34}. One of the specificity of these blocks is that they
		may (under sufficient conditions) be \textit{stacked} on top of each other to produce full orbits.
		\begin{figure}[h!]
			\begin{center}
				\scalebox{0.5}{\includegraphics{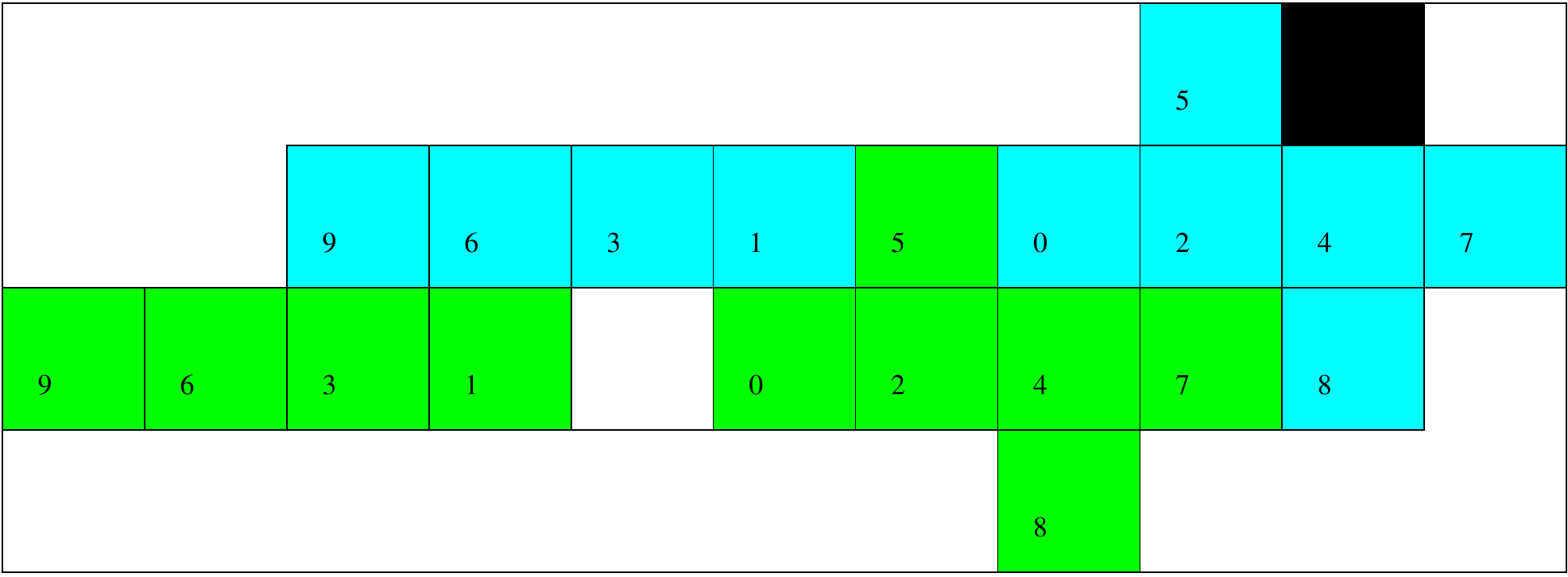}}
			\end{center}
			\vspace{-4mm}
			\caption{Two relative periods for slope $\frac{3}{34}$.}
			\label{fig:03_34}
		\end{figure}
		Our main result is to give necessary (theorem \ref{theorem:necessary_conditions}) and sufficient (theorem \ref{theorem:sufficient_conditions})
		conditions for a slope to have its orbit composed exclusively of elementary blocks, resulting in a relatively periodic orbit.
		We will say the slope is \textit{elementary} in that case.\\

		The article is divided into three parts. The general setup of the system is described in section \ref{sec:definitions}.
		Using unit squares for bricks allows our orbits to benefit from the properties of cutting sequences (see \cite{MH}).
		The inherent properties of these sequences are recounted in section \ref{subsec:slicing_sequence} and
		we introduce \textit{slicing sequences} as modified cutting sequences. We define elementary blocks
		in section \ref{subsec:elementary_blocks} and explain how they can be stacked on top of each other in section \ref{subsec:stacking}.

		Section \ref{sec:all_elementary_slopes} is the main part of the article. In section \ref{subsec:necessary_conditions},
		we give necessary conditions for a slope to be elementary (theorem \ref{theorem:necessary_conditions}). These conditions
		are mainly obtained by applying the properties of slicing sequences to elementary blocks.
		The approach (the use of slicing sequences) has two consequences.
		The first one is that it essentially reduces the search for elementary slopes to the study of slopes of the form $\frac{1}{\chi}$ for some positive integer $\chi$.
		This is developed in section \ref{subsec:finding_elementary_blocks}. The second consequence is the restriction of the number of the pairs of blocks
		whose stackings we need to examine. This is done in section \ref{subsec:stackability} and allows us to express sufficient conditions for elementary slopes
		in section \ref{subsec:sufficient_conditions}. These sufficient conditions are linear and quadratic diophantine equations.
		The results of this section can be summed up as follows.
		\begin{Theorem}
			All elementary slopes are rational. If $S=\frac{p}{q}$ (irreducible) is elementary, then $p$ is either
			$1$ or a multiple of $3$, and if $p\ne 1$, then $q$ is congruent to either $1$ or $p-1$ modulo $p$.

			Deciding if a slope is elementary comes down to checking if at most $2$ blocks are elementary and stack on top
			of each other.
		\end{Theorem}

		The last section touches on the subject of symbolic orbits, and is another attempt to draw parallels with billiards.
		To any orbit in a billiard, we can associate a (right-)infinite word (the symbolic orbit)
		$\omega = (\omega_n)_{n\in \mathds{N}}$ of $\{a, b\}^\mathds{N}$ such that $\omega_n = a$
		if the $n$th bounce happens on a horizontal edge and $\omega_n = b$ otherwise. These words
		are actually cutting sequences of lines (this is seen by \textit{unfolding} billiards, see \cite{Tab})
		and we refer again to \cite{MH} for their properties. The main point is that the words themselves tell us
		everything we need to know about the geometrical orbit. In particular the slope can simply be deduced by looking at
		the frequencies of $a$ and $b$ and a periodic word is equivalent to a rational slope.

		A similar word can be associated to any breakout orbit, where every letter of the word corresponds to a
		bounce and a brick being destroyed. However, what information this word gives us on the orbit is unclear.
		For example, relative peridiodicity implies that the symbolic orbit is periodic,
		but we are unable to say whether the converse is true. We argue in section \ref{sec:symbolic_orbits}
		that they are not good representations of the geometrical orbits by giving an infinity of elementary slopes all sharing
		the same symbolic orbit.

		One may also question the properties of the words obtained this way. We use elementary slopes to
		build periodic symbolic orbits with arbitrarily long periods.

	\section{Definitions}\label{sec:definitions}
		Throughout the article, we consider $\mathds{R}^2$ to be the domain and
		$\mathcal{O} = \{[x, x+1]\times [y, y+1], (x, y)\in \mathds{Z}^2\setminus \{(-1, 0)\}\}$
		to be the initial set of bricks.
		We say a ball travels \textbf{northeast} if its direction vector has two positive coordinates and \textbf{northwest}
		if the first coordinate of its direction vector is negative while the second one is positive.
		We consider all balls to be launched
		\begin{itemize}
			\item from initial position $(0, y)$ for some $0 < y < 1$,
			\item with initial direction northeast.
		\end{itemize}
		In effect, this means the first broken brick $[0, 1]\times [0, 1]$ is hit on a vertical edge after the ball has traveled a distance of $0$.
		We call \textbf{slope} the slope of the line carrying the ball's initial direction (for a distance of $0$) and
		we assume all slopes to be in the interval $(0, 1)$. Defining a slope as $\frac{p}{q}$ will always
		imply that $p$ and $q$ are both positive integers and that the fraction is irreducible.\\

		We recall here the definition of relative periodicity given in \cite{BF}. A \textit{half-strip} of width $w$ is defined
		as a part of $\mathds{R}^2$ whose points are at (Euclidian) distance less than $w$ from a half-line. For any given orbit,
		there exists a minimal $K\in \mathds{N}\cup \{+\infty\}$ and a set $(H_n)_{n\in [0, K)}$ of half-strips that contains all the bricks
		destroyed by the ball. It is proven in \cite{BF} that $K$ is either $1, 2$ or $+\infty$.
		For any $k\in [0, K)$, we define $(z^k_n)_{n\in \mathds{N}}$ as the sequence of $\mathds{Z}^2$ coordinates of (lower left corners of) bricks
		of $H_k$ (a brick may belong to multiple half-strips) destroyed by the ball. We then say that the orbit is \textbf{relatively periodic}
		if for any $k\in [0, K)$, the sequence $(z^k_n - z^k_{n-1})_{n\in \mathds{N}^*}$ is periodic or preperiodic (periodic after a certain index).

		Figure \ref{fig:01_02} is an example of a relatively periodic orbit where two half-strips are needed to cover all destroyed bricks,
		while figure \ref{fig:03_34} is an example requiring a single half-strip.\\

		We will consider that hitting a brick on a horizontal edge applies a symmetry, along the axis the edge lives on,
		to both the remaining obstacles and the ball's trajectory. This is similar to billiard's \textit{unfolding} (see \cite{Tab}),
		but we only apply it to horizontal edges because we want to retain the orbits' abscissas. This makes the ball's ordinate strictly increasing,
		softening calculations and formulas while having no effect on relative periodicity. Keep in mind that this setting cannot be represented on pictures.

		In that setting, we associate, to any orbit that starts at the edge of the first broken brick, the \textbf{traveling function}
		\begin{center}
			$\begin{array}{cccc}
				T: & \mathds{N} & \to & \mathds{Z}\times \mathds{R}^+\\
				& n & \mapsto & (T_x(n), T_y(n))
			\end{array}$
		\end{center}
		where $n$ represents the \textit{horizontal} distance traveled by the ball.
		We will call $T_x(n)$ and $T_y(n)$ the \textbf{traveling abscissa} and \textbf{traveling ordinate} respectively.

		\begin{Property}
			As a consequence of the setting described above, we have, for any slope $S$ and any $n\in \mathds{N}$,
			\[T_y(n) = T_y(0) + Sn\]
		\end{Property}

		We will always assume that the orbits we choose do not intersect any point of $\mathds{Z}^2$.

		\subsection{The slicing sequence}\label{subsec:slicing_sequence}
			The \textit{cutting sequence} of an orbit is defined as the infinite word of $\{a, b\}^\infty$ obtained by moving
			along the orbit and writing $b$ when the orbit reaches a point with integral abscissa  and $a$ when it reaches
			a point with integral ordinate. We emphasize that it is different from the symbolic orbit (see section \ref{sec:symbolic_orbits})
			because the letters of the symbolic orbit correspond only to bricks being hit, while the cutting sequence also considers vertical
			and horizontal lines being crossed.

			Such a sequence is known to be \textit{balanced}, that is, if two finite words of the cutting sequence have the same length,
			then their numbers of $b$ (or $a$) are at most $1$ apart. We refer to \cite{MH} for an in-depth analysis
			of cutting and balanced sequences, with a few notes. The vocabulary has changed since \cite{MH}:
			\textit{Sturmian} is used today to refer to \cite{MH}'s \textit{irrational Sturmian}
			and \textit{balanced} is used today to refer to \cite{MH}'s \textit{rational and irrational Sturmian}.
			The \textit{skew} issue is avoided here by considering that our trajectories do not intersect $\mathds{Z}^2$.

			Given a slope $S$ and an initial ordinate $T_y(0)\in (0, S)$, we define the \textbf{slicing sequence} $(\chi_n)_{n\in \mathds{N}}$
			as the sequence of numbers of consecutive $b$ in the cutting sequence. Alternatively,
			\begin{align*}
				\forall N&\in \mathds{N}, & T_y(0) + (\sum\limits_{0\le n\le N} \chi_n) S&\in (N+1, N+1+S)
			\end{align*}
			Note that different initial ordinates may produce different slicing sequences for the slope $S$. All of these slicing sequences
			share common properties. The following property is given in \cite{MH} (section $4$).
			\begin{Property}\label{property:slicing_elements}
				All slicing sequences of $S$ contain at most two elements $\lambda_0$ and $\lambda_0+1$ where
				\begin{align*}
					1 &= \lambda_0 S + \mu_0 & \lambda_0\in \mathds{N}, \mu_0&\in (0, S)
				\end{align*}
			\end{Property}

			Additionally, from \cite{MH} (theorem $8.1$), the slicing sequences are also balanced and the relative frequencies of
			$\lambda_0$ and $\lambda_0+1$ is given by
			\begin{align*}
				\lim\limits_{n\to +\infty} \frac{\#\{\chi_j = \lambda_0, 0\le j < n\}}{\#\{\chi_j = \lambda_0+1, 0\le j < n\}} &= \frac{\mu_0}{S - \mu_0}
			\end{align*}
			We will refer to the most frequent and less frequent terms of the slicing sequences as the \textbf{leading slice} and \textbf{correcting slice}
			of $S$ respectively.
			\begin{Property}\label{property:slicing_gaps}
				(Again from \cite{MH}.) The number of consecutive leading slices in any slicing sequence is either $\lambda_1$ or $\lambda_1+1$ with
				\begin{align*}
					1 &= \lambda_1 S' + \mu_1 & S' &= min(\frac{S - \mu_0}{\mu_0}, \frac{\mu_0}{S - \mu_0})
				\end{align*}
			\end{Property}

		\subsection{Elementary blocks}\label{subsec:elementary_blocks}
			\subsubsection{Ternary blocks}\label{subsubsec:ternary_blocks}
				We will use figure \ref{fig:01_41} as a visual aid to explain notations. The figure is obtained by launching
				a ball with slope $\frac{1}{41}$ from any position $(0, T_y(0))$ with $T_y(0)\in (0, \frac{1}{41})$.

				\begin{figure}[h!]
					\begin{center}
						\scalebox{0.6}{\includegraphics{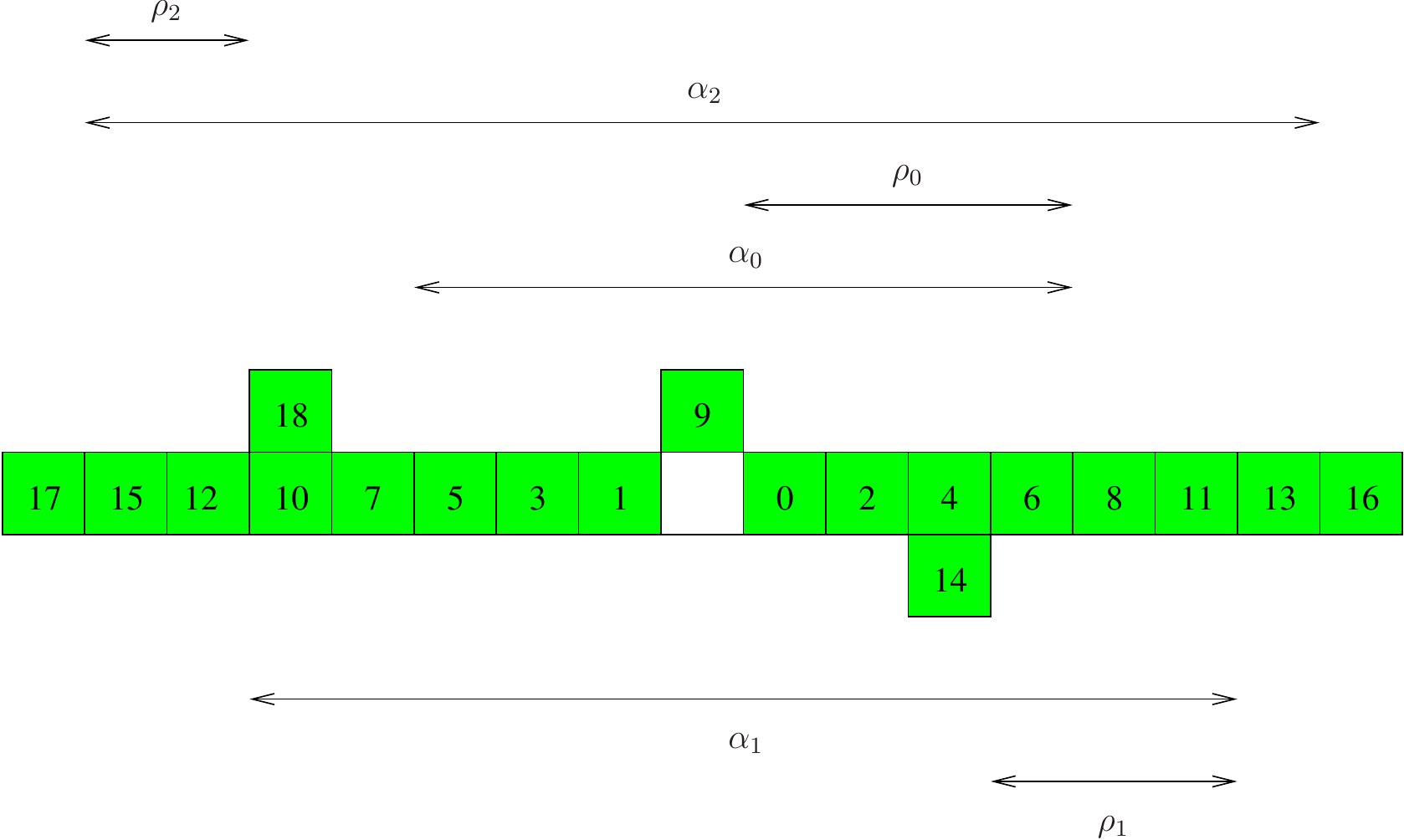}}
					\end{center}
					\vspace{-4mm}
					\caption{Ternary block associated with $\frac{1}{41}$.}
					\label{fig:01_41}
				\end{figure}

				For any pair $(\alpha, \rho)\in \mathds{N}\times \mathds{N}$ with $0\le \rho < \alpha+1$, we define
				\begin{align*}
					|\alpha, \rho| &= \frac{\alpha(\alpha+1)}{2} + (\rho + 1)
				\end{align*}
				and
				\begin{align*}
					X(|\alpha, \rho|) &= \left( \sum\limits_{0 < n \le \alpha} (-1)^n n \right) + (-1)^{\alpha+1} (\rho + 1)
				\end{align*}

				\noindent
				Interpreting these definitions on figure \ref{fig:01_41}, $|\alpha_i, \rho_i|$ ($0\le i\le 2$) corresponds to the ceiled (rounded up to an integer)
				horizontal distance required to reach the $i$th (starting at $0$) horizontal edge and $X(|\alpha_i, \rho_i|)$ ($0\le i\le 2$)
				is the abscissa of the ball after traveling a horizontal distance of $|\alpha_i, \rho_i|$.\\

				We define a \textbf{ternary block} as any tuple $\Delta = ((\alpha_i, \rho_i))_{0\le i\le 2}$
				such that
				\begin{itemize}
					\item $(\alpha_i, \rho_i)\in \mathds{N}\times \mathds{N}$ and $0\le \rho_i < \alpha_i+1$ for any $0\le i\le 2$,
					\item there exists a slope $S$ such that there exists an ordinate $0 < T_y(0) < S$ for which
						\begin{equation}\label{eq:slope_existence}
							(i+1) < T_y(|\alpha_i, \rho_i|) < (i+1) + S
						\end{equation}
						holds for any $0\le i\le 2$. Such a slope is said to be $\Delta$-\textbf{primary} or \textbf{primary for} $\Delta$.
				\end{itemize}

				\noindent
				In effect, this means that there exists an orbit with initial ordinate in $(0, S)$ such that the $i$th, for any $0\le i\le 2$, horizontal edge is reached after traveling
				a horizontal distance $x_i$ with $|\alpha_i, \rho_i| - 1 < x_i < |\alpha_i, \rho_i|$. Any such orbit is called a \textbf{$\Delta$-orbit}.

				The $0$th brick hit on a horizontal edge is called the \textbf{top brick} and the $1$th is called the \textbf{bottom brick}.
				The last encountered horizontal edge may belong to the (now broken) top brick (see figure \ref{fig:03_34})
				or to another (not yet broken) brick (as is the case on figure \ref{fig:01_41}).\\

				For any ternary block $\Delta = ((\alpha_i, \rho_i))_{0\le i\le 2}$, we define
				\begin{align*}
					|\Delta| &= |\alpha_2, \rho_2|
				\end{align*}

				\begin{Property}\label{property:coincidence}
					If $\Delta$ is a ternary block, then $X$ and $T_x$ coincide on $|\alpha_i, \rho_i|$ for any $0\le i\le 2$ (and therefore on $|\Delta|$)
					for any $\Delta$-orbit.
				\end{Property}

			\subsubsection{Ternary blocks and slicing sequence}
				We also define ternary blocks as triplets of elements of the slicing sequences. We will write
				$\Delta = (\chi_0, \chi_1, \chi_2) = ((\alpha_i, \rho_i))_{0\le i\le 2}$ with
				\begin{align*}
					\chi_0 &= |\alpha_0, \rho_0|\\
					\chi_1 &= |\alpha_1, \rho_1| - |\alpha_0, \rho_0|\\
					\chi_2 &= |\alpha_2, \rho_2| - |\alpha_1, \rho_1|
				\end{align*}
				Equations \ref{eq:slope_existence} can then be rewritten, for any $0\le i\le 2$, as
				\begin{equation}\label{eq:slope_existence_chi}
					(i+1) < T_y(0) + (\sum\limits_{0\le k\le i}\chi_k) S < (i+1) + S
				\end{equation}
				Note that the expression in terms of elements of the slicing sequences gives no information regarding the abscissas of the bricks
				that constitute the ternary block, which explains the need for the two definitions.

				From property \ref{property:slicing_elements}, the ternary blocks a slope can be primary for will have all of their coordinates
				within one of each other. As such, we will usually pick a representative $\Delta = ((\alpha_i, \rho_i))_{0\le i\le 2}$
				and express the other blocks using the following notations.

				For any pair $(\alpha, \rho)\in \mathds{N}\times \mathds{N}$ with $0\le \rho < \alpha+1$, we define
				\begin{align*}
					d_-(\alpha, \rho) &= (\alpha, \rho-1) & \text{if }\rho > 0\\
					d_-(\alpha, 0) &= (\alpha-1, \alpha-1) &
				\end{align*}
				and
				\begin{align*}
					d_+(\alpha, \rho) &= (\alpha, \rho+1) & \text{if }\rho < \alpha\\
					d_+(\alpha, \alpha) &= (\alpha+1, 0) &
				\end{align*}

				\begin{Remark}
					As many arguments will be common to both $d_-$ and $d_+$,
					we will often use the symbol $\ast\in \{-, +\}$ and only revert to $-$ and $+$
					when necessary. The symbol will also be used in equations as both a unary and binary operator.
				\end{Remark}

			\subsubsection{Elementary blocks}
				Finally, we define an \textbf{elementary block} as a ternary block $\Delta$ such that
				\begin{itemize}
					\item if $\alpha_0$ and $\alpha_2$ have the same parity, then
						\begin{equation}\label{eq:top_same_parity}
							\alpha_2 - 2\rho_2 = \alpha_0 - 2\rho_0
						\end{equation}
					\item and if $\alpha_0$ and $\alpha_2$ have the different parities, then
						\begin{equation}\label{eq:top_different_parities}
							\alpha_2 - 2\rho_2 = -\alpha_0 + 2\rho_0 + 1
						\end{equation}
				\end{itemize}

				\noindent
				These conditions ensure that the third encountered horizontal edge belongs to the (now broken) top brick.
				The following proposition will be used throughout this article.
				\begin{Proposition}\label{prop:elementary_rho2}
					If $\Delta = ((\alpha_i, \rho_i))_{0\le i\le 2}$ is an elementary block, then $\rho_2\notin \{0, \alpha_2\}$.
				\end{Proposition}

				\begin{proof}
					If $\alpha_0$ and $\alpha_2$ have the same parity, then equation \ref{eq:top_same_parity} applies and we have
					\begin{align*}
						\alpha_2 &= \alpha_0 - 2\rho_0 & \text{if}~ \rho_2 &= 0\\
						-\alpha_2 &= \alpha_0 - 2\rho_0 & \text{if}~ \rho_2 &= \alpha_2
					\end{align*}
					With $0\le \rho_0\le \alpha_0\le \alpha_2$, we deduce, in both cases, $\rho_0 = \rho_2$ and $\alpha_0 = \alpha_2$, which is impossible for slopes in $(0, 1)$.

					If $\alpha_0$ and $\alpha_2$ have different parities, then equation \ref{eq:top_different_parities} applies and we have
					\begin{align*}
						\alpha_2 &= -\alpha_0 + 2\rho_0 + 1 & \text{if}~ \rho_2 &= 0\\
						-\alpha_2 &= -\alpha_0 + 2\rho_0 + 1 & \text{if}~ \rho_2 &= \alpha_2
					\end{align*}
					Again with $0\le \rho_0\le \alpha_0\le \alpha_2$, we deduce that the case $\rho_2=\alpha_2$ is impossible
					and that if $\rho_2=0$, we must have $\rho_0 = \alpha_0$ and $\alpha_2 = \alpha_0+1$. In that case,
					$|\alpha_2, \rho_2| = |\alpha_0, \rho_0|+1$ and the only possible values for $|\alpha_1, \rho_1|$ are either
					$|\alpha_0, \rho_0|$ or $|\alpha_2, \rho_2|$, neither of which is possible for slopes in $(0, 1)$.
				\end{proof}

		\subsection{Stacking elementary blocks}\label{subsec:stacking}
			In order to get full orbits, we want to be able to stack elementary blocks on top of each other,
			with the top brick of a given elementary block acting as the origin square of the following elementary block.

			Given two elementary blocks $\Delta = ((\alpha_i, \rho_i))_{0\le i\le 2}$ and $\Delta' = ((\alpha_i', \rho_i'))_{0\le i\le 2}$,
			we say that $\Delta'$ \textbf{stacks on top of} $\Delta$ if there exists an orbit such that for any $0\le i\le 2$, we have
			\begin{align*}
				T_x(|\Delta| + |\alpha_i', \rho_i'|) &= T_x(|\Delta|) + (-1)^{\alpha_2+1} X(|\alpha_i', \rho_i'|)
			\end{align*}

			\noindent
			This relation is neither commutative nor reflexive. An elementary block $\Delta$ that stacks on top of itself is said to be \textbf{self-stackable}.
			Figure \ref{fig:03_34} shows an example of a self-stackable elementary block depicted twice: in green then in blue.

			It should be noted that the geometry of a ternary block is not uniquely defined and depends on the initial direction of the ball.
			We still assume that the ball initially launches northeast, which takes care of the first elementary block, but we can make no
			such assumption for the stacked blocks. This fact is the origin of the $(-1)^{\alpha_2+1}$ term in the equation above.

			Figure \ref{fig:01_138} is another example of a self-stackable elementary block. This time however, the blue depiction is a
			symmetrical image of the green depiction. This is caused by the ball traveling northwest while starting the second block.
			\begin{figure}[h!]
				\begin{center}
					\scalebox{0.6}{\includegraphics{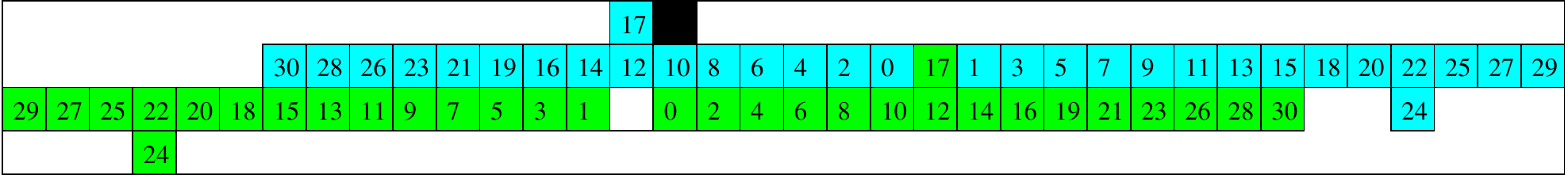}}
				\end{center}
				\vspace{-4mm}
				\caption{One relative period for slope $\frac{1}{138}$ ($\alpha_2$ even).}
				\label{fig:01_138}
			\end{figure}

			A necessary condition for the existence of an orbit going through $\Delta$ and $\Delta'$ successively
			is that the bricks broken during the $\Delta$ phase (horizontal distance traveled is $< |\Delta|$)
			do not overlap with the bricks broken during the $\Delta'$ phase. The condition is not sufficient,
			as $\Delta$ and $\Delta'$ could have a disjoint set of primary slopes. Even with common primary
			slopes, the initial ordinate of the ball when the first brick of $\Delta'$ is reached could be incompatible with a $\Delta'$-orbit.
			These sufficient conditions will be ignored with no repercussions.

			The only brick that requires attention is the bottom brick of $\Delta'$.
			We focus on the left edge of the bottom brick of $\Delta'$ and the left edges of the leftmost and rightmost bricks of $\Delta$.
			If a ball is traveling east (increasing abscissa) when hitting the bottom brick of $\Delta'$, then the traveling abscissa at $|\Delta| + |\alpha_1', \rho_1'|$
			points to the right edge of the bottom brick, and we need to subtract $1$. Hence, we define $\zeta(\Delta, \Delta') = 1$ if $\alpha_2$ and $\alpha_1'$
			have the same parity and $0$ otherwise.

			The necessary condition for $\Delta'$ to be stackable on top of $\Delta$ is then:
			\begin{itemize}
				\item if $\alpha_2$ is even, then
					\begin{align*}
						T_x(|\Delta| + |\alpha_1', \rho_1'|) - \zeta(\Delta, \Delta') &\notin [X(|\alpha_2-1, 0|) - 2, X(|\alpha_2, 0|) + 1]\\
						&\notin \left[-\frac{\alpha_2}{2} - 1, \frac{\alpha_2}{2}\right]
					\end{align*}

				\item and if $\alpha_2$ is odd, then
					\begin{align*}
						T_x(|\Delta| + |\alpha_1', \rho_1'|) - \zeta(\Delta, \Delta') &\notin [X(|\alpha_2, 0|) - 2, X(|\alpha_2-1, 0|) + 1]\\
						&\notin \left[-\frac{\alpha_2+1}{2}-1, \frac{\alpha_2-1}{2}\right]
					\end{align*}
			\end{itemize}

			\noindent
			One may look at figures \ref{fig:03_34}, \ref{fig:01_138} and \ref{fig:01_297} for visual support.
			\begin{figure}[h!]
				\begin{center}
					\scalebox{0.6}{\includegraphics{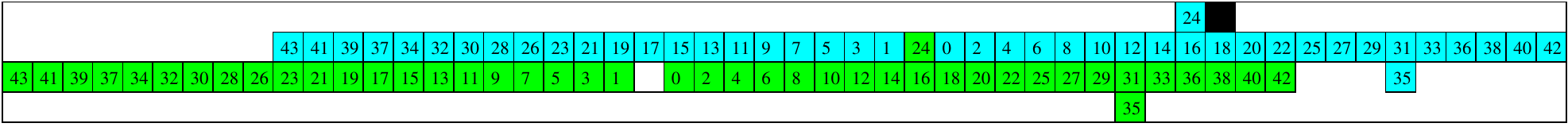}}
				\end{center}
				\vspace{-4mm}
				\caption{Two relative periods for slope $\frac{1}{297}$ ($\alpha_2$ odd, $\alpha_1'$ odd).}
				\label{fig:01_297}
			\end{figure}

		\subsection{Elementary slopes}
			Let $(\Delta_n = ((\alpha_{i, n}, \rho_{i, n}))_{0\le i\le 2})_{n\in \mathds{N}}$ be a sequence of elementary blocks.
			For any $N\in \mathds{N}$, define
			\begin{align*}
				\alpha_N &= \sum\limits_{0\le j < N} (\alpha_{2, j}+1) & \Sigma_N &= \sum\limits_{0\le j < N} |\Delta_j|
			\end{align*}

			A given slope $S$ will be said to be \textbf{elementary} with \textbf{elementary sequence} $(\Delta_n)_{n\in \mathds{N}}$
			if there exists an origin point with ordinate in $(0, S)$ and a sequence of elementary blocks such that
			\begin{align*}
				T_x(\Sigma_N + |\alpha_{i, N}, \rho_{i, N}|) &= T_x(\Sigma_N) + (-1)^{\alpha_N} X(|\alpha_{i, N}, \rho_{i, N}|)
			\end{align*}

			\noindent
			for any $0\le i\le 2$ and any $N\in \mathds{N}$. This is essentially saying that there exists an orbit associated with the slope
			that successively goes through each block of its elementary sequence.

			We note that a slicing sequence is simply a less detailed expression of an elementary sequence.

	\section{Characterising all elementary slopes}\label{sec:all_elementary_slopes}
		Picking a slope $S$ and an origin ordinate $T_y(0)\in (0, S)$ (carefully chosen so the orbit does not intersect $\mathds{Z}^2$)
		is enough to fully determine a slicing sequence, which in turn is enough to determine a sequence of ternary blocks $(\Delta_n)_{n\in \mathds{N}}$.
		Whether this sequence is elementary is the result of two factors:
		\begin{itemize}
			\item for any $n\in \mathds{N}$, $\Delta_n$ is an elementary block,
			\item for any $n\in \mathds{N}^*$, $\Delta_n$ stacks on top of $\Delta_{n-1}$.
		\end{itemize}

		We deal with these issues separately. In section \ref{subsec:necessary_conditions}, we study elementary sequences and obtain
		necessary conditions for their existence. This, in turn imposes necessary conditions on slopes. The issue of the origin ordinate
		is mostly ignored there.

		Sections \ref{subsec:finding_elementary_blocks} and \ref{subsec:stackability} act together to provide sufficient conditions
		for slopes to be elementary. We explain how studying slopes of the form $\frac{1}{\chi}$ is enough
		to not only determine all other elementary slopes, but also give the form of their elementary sequence.
		The various forms of elementary sequences are given in lemma \ref{lemma:necessary_conditions}.
		In section \ref{subsec:finding_elementary_blocks}, we explain how these elementary sequences
		are obtained, and we study the stackability of elementary blocks in section \ref{subsec:stackability}.

		\subsection{Necessary conditions for elementary slopes}\label{subsec:necessary_conditions}
			The aim of this section is to prove the following theorem.
			\begin{Theorem}\label{theorem:necessary_conditions}
				All elementary slopes are rational. If $S=\frac{p}{q}$ (irreducible) is elementary, then $p$ is either
				$1$ or a multiple of $3$, and if $p\ne 1$, then $q$ is congruent to either $1$ or $p-1$ modulo $p$.
			\end{Theorem}

			The proof involves studying the blocks that constitute elementary sequences. Throughout this section,
			we consider a slope $S$ with slicing sequence $(\chi_n)_{n\in \mathds{N}}$ (for some initial ordinate in $(0, S)$),
			define $\chi$ and $\chi'$, respectively, as the leading and correcting slices of $S$ and define
			\begin{align*}
				\Lambda &= (\chi, \chi, \chi) & \Lambda_{02} &= (\chi', \chi, \chi')\\
				\Lambda_0 &= (\chi', \chi, \chi) & \Lambda_1 &= (\chi, \chi', \chi) & \Lambda_2 &= (\chi, \chi, \chi')
			\end{align*}
			A correcting slice cannot occur twice consecutively by definition. Together with property \ref{property:slicing_gaps},
			we conclude that these five elementary blocks are the only blocks that could be part of an elementary sequence.
			We prove the following lemma, which implies theorem \ref{theorem:necessary_conditions}.
			\begin{Lemma}\label{lemma:necessary_conditions}
				All elementary sequences are periodic. If an elementary sequence has minimal period's length $1$,
				then its period is one of $\{\Lambda, \Lambda_0, \Lambda_1, \Lambda_2\}$.
				If an elementary sequence has minimal period's length $>1$, then its period is, up to circular permutation,
				equal to $\Lambda^k\Lambda_0$ for some integer $k > 0$.
			\end{Lemma}

			Again, since there cannot be two consecutive $\chi'$ in the slicing sequence, $\Lambda_{02}$ cannot be the only element of the elementary sequence.
			We now assume the elementary sequence has at least two distinct elements. The strategy for proving lemma \ref{lemma:necessary_conditions}
			is to highlight pairs of blocks which cannot be elementary together and to show that at least one of these pairs necessarily appears
			in each case not covered by lemma \ref{lemma:necessary_conditions}.

			\begin{Proposition}\label{prop:not_both_elementary}
				If $(\Delta, \Delta')$ is an element of
				\begin{align*}
					\{(\Lambda, \Lambda_1), (\Lambda, \Lambda_2), (\Lambda_0, \Lambda_{02}), (\Lambda_1, \Lambda_{02}), (\Lambda_2, \Lambda_{02})\}
				\end{align*}
				then $\Delta$ and $\Delta'$ cannot be elementary together.
			\end{Proposition}

			\begin{proof}
				The argument for the first three pairs is the same. Assume
				\[(\Delta, \Delta')\in \{(\Lambda, \Lambda_1), (\Lambda, \Lambda_2), (\Lambda_0, \Lambda_{02})\}\]
				define $\Delta = ((\alpha_i, \rho_i))_{0\le i\le 2}$ and $\Delta' = ((\alpha_i', \rho_i'))_{0\le i\le 2}$
				and observe that $\Delta'$ is such that we have, for some $\ast\in \{-, +\}$,
				\begin{align*}
					(\alpha_0', \rho_0') &= (\alpha_0, \rho_0) & (\alpha_2', \rho_2') &= d_\ast(\alpha_2, \rho_2)
				\end{align*}

				\noindent
				We possibly swap $\Delta$ and $\Delta'$ in order to be able to assume $(\alpha_2', \rho_2') = d_-(\alpha_2, \rho_2)$.
				According to equations \ref{eq:top_same_parity} and \ref{eq:top_different_parities}, $\Delta$ and
				$\Delta'$ can only be elementary together if $\rho_2 = 0$, which is impossible from proposition \ref{prop:elementary_rho2}.\\

				What follows is true for any
				\[(\Delta, \Delta')\in \{(\Lambda_1, \Lambda_{02}), (\Lambda_2, \Lambda_{02})\}\]
				We assume $\Delta$ is elementary and show that $\alpha_0$ and $\alpha_2$ must then have different parities.
				Set $\Delta = ((\alpha_i, \rho_i))_{0\le i\le 2}$ and observe that
				\begin{align*}
					\frac{\alpha_0(\alpha_0+1)}{2} + \rho_0 + 1 &= \chi\\
					\frac{\alpha_2(\alpha_2+1)}{2} + \rho_2 + 1 &= 3\chi \ast 1
				\end{align*}
				for some $\ast\in \{-, +\}$. If $\alpha_0$ and $\alpha_2$ have the same parity,
				then equation \ref{eq:top_same_parity} applies and we have $\alpha_2 - 2\rho_2 = \alpha_0 - 2\rho_0$.
				With the equations above, we deduce
				\[\frac{(\alpha_2+1)^2 - (\alpha_0+1)^2}{2} = 2\chi \ast 1\]
				which is impossible since $\alpha_0$ and $\alpha_2$ having the same parity means the left hand side must be a multiple of $2$.

				We conclude that equation \ref{eq:top_different_parities} applies and we have $\alpha_2 - 2\rho_2 = -\alpha_0 + 2\rho_0 + 1$.
				With $\Lambda_{02} = ((\alpha_i', \rho_i'))_{0\le i\le 2}$, observe that we have, for some $\ast\in \{-, +\}$,
				\begin{align*}
					(\alpha_0', \rho_0') &= d_\ast(\alpha_0, \rho_0) & (\alpha_2', \rho_2') &= d_\ast(\alpha_2, \rho_2)
				\end{align*}
				From proposition \ref{prop:elementary_rho2}, $\rho_2\notin \{0, \alpha_2\}$, and $\Lambda_{02}$ can only be elementary if
				\begin{align*}
					\alpha_2 - 2(\rho_2\ast 1) &= -\alpha_0 + 2(\rho_0\ast 1) + 1 & \text{if}~ &(\alpha_0', \rho_0') = (\alpha_0, \rho_0\ast 1)\\
					\alpha_2 - 2(\rho_2 - 1) &= (\alpha_0-1) - 2(\alpha_0-1) & \text{if}~ &(\rho_0 = 0)\\
					&&\text{and}~ &(\alpha_0', \rho_0') = (\alpha_0-1, \alpha_0-1)\\
					\alpha_2 - 2(\rho_2 + 1) &= (\alpha_0+1) - 2(0) & \text{if}~ &(\rho_0 = \alpha_0)\\
					&&\text{and}~ &(\alpha_0', \rho_0') = (\alpha_0+1, 0)
				\end{align*}
				all of which contradict $\alpha_2 - 2\rho_2 = -\alpha_0 + 2\rho_0 + 1$.
			\end{proof}

			\begin{Proposition}\label{prop:periodic_elementary_sequence}
				An elementary slope is rational and has a periodic elementary sequence.
			\end{Proposition}

			\begin{proof}
				Suppose $S\in (0, 1)$ is an irrational slope that is elementary for some
				initial ordinate $T_y(0)$ with $T_y$ being its traveling ordinate. From the irrationality of $S$, we have
				\[\overline{\{T_y(3k) \pmod{1}, k\in \mathds{N}\}} = [0, S]\]
				with the overline denoting the Euclidian metric completion. This tells us the slope must be primary for some elementary block
				for any initial ordinate.

				We define $y_i$, for any $0\le i\le 2$, as the unique element of $(0, S)$ such that
				\[y_i = (i+1) - N_i S\]
				for some integer $N_i$ and observe that $y_0\ne y_2$, again from irrationality.

				Pick two initial ordinates $y$ and $y'$ such that $y < y_2 < y'$ and $y_0\notin [y, y']$.
				The blocks $\Delta, \Delta'$ such that $S$ is $\Delta$-primary and $\Delta'$-primary for initial ordinates $y$ and $y'$ respectively
				have the same index $0$ coordinate but different index $2$ coordinate. From proposition \ref{prop:not_both_elementary},
				they cannot be elementary together and we conclude that $S$ must be rational.\\

				Picking an elementary slope $\frac{p}{q}$ and noting $T_y$ its traveling ordinate, proving the elementary sequence is periodic
				comes down to finding an integer $k$ for which $T_y(k)-T_y(0)$ is an integer divisible by $3$. We can simply choose $k=3q$.
			\end{proof}

			A direct consequence of this proposition is that the set of elementary orbits is a subset of relatively periodic orbits
			(see section \ref{sec:definitions} for the definition). We also note that all elementary orbits can be covered by a single half-strip.

			\begin{Proposition}
				Neither $\Lambda_1 = (\chi, \chi', \chi)$ nor $\Lambda_2 = (\chi, \chi, \chi')$ nor $\Lambda_{02} = (\chi', \chi, \chi')$
				can be an element of an elementary sequence with minimal period's length $>1$.
			\end{Proposition}

			\begin{proof}
				We first observe that as a consequence of property \ref{property:slicing_gaps}, the element following $\Lambda_{02}$
				is either $\Lambda_1$ or $\Lambda_2$ in any elementary sequence. We then only need to prove the proposition for $\Lambda_1$ and $\Lambda_2$.
				With \ref{prop:not_both_elementary}, all we need to prove is that an elementary sequence with minimal period's length $>1$
				that contain either $\Lambda_1$ or $\Lambda_2$ necessarily contains either $\Lambda$ or $\Lambda_{02}$.

				We define $\lambda_1$ as in section \ref{subsec:slicing_sequence}.
				The argument is given by property \ref{property:slicing_gaps}: the number of consecutive $\chi$ is either $\lambda_1$ or $\lambda_1+1$.
				We first assume $(\chi_n)_{n\in \mathds{N}}$ contains $\Lambda_1$. Observe that $\lambda_1\ge 1$ by definition.
				\begin{itemize}
					\item If $\lambda_1=1$, there exists an integer $k\ge 0$ such that
						\[(\chi, \chi', \chi) (\chi', \chi, \chi)^k (\chi', \chi, \chi') = \Lambda_1 (\chi', \chi, \chi)^k \Lambda_{02}\]
						is a finite subsequence of $(\chi_n)_{n\in \mathds{N}}$.

					\item If $\lambda_1=2$, the number of consecutive $\chi$ cannot stay constant at $2$,
						otherwise the elementary sequence has period $1$. We deduce there exists an integer $k > 0$ such that
						\[(\chi, \chi', \chi) (\chi, \chi, \chi')^k (\chi, \chi, \chi) = \Lambda_1 (\chi, \chi, \chi')^k \Lambda\]
						is a finite subsequence of $(\chi_n)_{n\in \mathds{N}}$.

					\item If $\lambda_1\ge 3$, then either
						\[(\chi, \chi', \chi) (\chi, \chi, \chi) = \Lambda_1 \Lambda\]
						or
						\[(\chi, \chi', \chi) (\chi, \chi, \chi') (\chi, \chi, \chi) = \Lambda_1 \Lambda_2 \Lambda\]
						is a finite subsequence of $(\chi_n)_{n\in \mathds{N}}$.
				\end{itemize}

				\noindent
				We now assume that $(\chi_n)_{n\in \mathds{N}}$ contains $\Lambda_2$. The number of consecutive $\chi$
				cannot stay constant at $2$, otherwise the elementary sequence has period $1$.
				\begin{itemize}
					\item If $\lambda_1=1$,
						\[(\chi, \chi, \chi') (\chi, \chi', \chi) = \Lambda_2 \Lambda_1\]
						is a finite subsequence of $(\chi_n)_{n\in \mathds{N}}$, and we go back to the previous case.

					\item If $\lambda_1\ge 2$,
						\[(\chi, \chi, \chi') (\chi, \chi, \chi) = \Lambda_2 \Lambda\]
						is a finite subsequence of $(\chi_n)_{n\in \mathds{N}}$.
				\end{itemize}
			\end{proof}

			It follows than an elementary sequence with minimal period's length $>1$ can only contain $(\chi, \chi, \chi)$ and $(\chi', \chi, \chi)$.
			The following property restrict the form of the period even further, and is a direct consequence of property \ref{property:slicing_gaps}.
			\begin{Property}
				If an elementary sequence has minimal period's length $>1$, then its minimal period only contains a single
				occurence of $(\chi', \chi, \chi)$.
			\end{Property}

			This concludes the proof of lemma \ref{lemma:necessary_conditions}.
			The single occurence of $(\chi', \chi, \chi)$ also means there is a single occurence of $\chi'$ in the slicing sequence's period,
			which means $q$ is congruent to either $1$ or $p-1$ modulo $p$, as stated in theorem \ref{theorem:necessary_conditions}.

		\subsection{Finding elementary blocks}\label{subsec:finding_elementary_blocks}
			The aim of this section is to give conditions for the blocks used in lemma \ref{lemma:necessary_conditions} to be elementary.
			All questions regarding stackability are kept for section \ref{subsec:stackability}.

			We note that a slope $\frac{1}{\chi}$ is necessarily primary for a single ternary block $\Delta = (\chi, \chi, \chi)$,
			where $\chi$ is the leading (and only) slice of $S$.
			For such a slope, we define, for any $\ast\in \{-, +\}$ and any positive integer $J$ the slopes
			\[\Gamma_{J\ast} = \frac{3J}{3J\chi\ast 1}\]
			and the ternary blocks
			\begin{align*}
				\Lambda &= (\chi, \chi, \chi) & \Lambda_{0\ast} &= (\chi\ast 1, \chi, \chi)\\
				\Lambda_{1\ast} &= (\chi, \chi\ast 1, \chi) & \Lambda_{2\ast} &= (\chi, \chi, \chi\ast 1)
			\end{align*}

			We want to give conditions on $\Lambda$ for $\Lambda, \Lambda_{0\ast}, \Lambda_{1\ast}$ or $\Lambda_{2\ast}$ to be elementary.
			As a first step, we must check that all of these are ternary blocks, meaning they accept primary slopes.
			We already know that $\frac{1}{\chi}$ is primary for $\Lambda$. For the other three,
			following the equations of \ref{eq:slope_existence_chi}, one may check the following property.
			\begin{Property}\label{property:all_ternary}
				For any $J\in \mathds{N}^*$, $\Gamma_{J-}$ is primary for $\Lambda, \Lambda_{2-}, \Lambda_{1-}, \Lambda_{0-}$
				respectively with any initial ordinates $y, y_{2-}, y_{1-}, y_{0-}$ such that
				\[0 <  y < \frac{3J-3}{3J\chi-1} < y_{2-} <  \frac{3J-2}{3J\chi-1} < y_{1-} < \frac{3J-1}{3J\chi-1} < y_{0-} < \frac{3J}{3J\chi-1}\]
				and $\Gamma_{J+}$ is primary for $\Lambda_{0+}, \Lambda_{1+}, \Lambda_{2+}, \Lambda$
				respectively with any initial ordinates $y_{0+}, y_{1+}, y_{2+}, y$ such that
				\[0 < y_{0+} < \frac{1}{3J\chi+1} < y_{1+} < \frac{2}{3J\chi+1} < y_{2+} < \frac{3}{3J\chi+1} < y < \frac{3J}{3J\chi+1}\]
				Neither $\Gamma_{1-}$ nor $\Gamma_{1+}$ is $\Lambda$-primary.
			\end{Property}

			We separate the various forms of elementary periods given in lemma \ref{lemma:necessary_conditions} into three groups.
			The first group only contains period $\Lambda$.
			\begin{Proposition}\label{prop:lambda_elementary}
				Block $\Lambda = (\chi, \chi, \chi) = ((\alpha_i, \rho_i))_{0\le i\le 2}$ is elementary
				if and only if $\alpha_2 - 2\rho_2 = \alpha_0 - 2\rho_0$.
			\end{Proposition}

			\begin{proof}
				Observe that we have, by definition,
				\begin{equation}\label{eq:lambda_elementary}
					\begin{aligned}
						\frac{\alpha_0(\alpha_0+1)}{2} + \rho_0 + 1 &= \chi\\
						\frac{\alpha_2(\alpha_2+1)}{2} + \rho_2 + 1 &= 3\chi
					\end{aligned}
				\end{equation}

				Assume $\Lambda$ is elementary and $\alpha_0$ and $\alpha_2$ have different parities. The equations
				above together with equation \ref{eq:top_different_parities} give
				\[\frac{(\alpha_2)^2 - (\alpha_0+1)^2 + 2}{2} + 2\rho_2 = 2\chi\]
				which is impossible as $\alpha_0$ and $\alpha_2$ having different parities implies the left hand side is odd.
				We deduce $\alpha_0$ and $\alpha_2$ have the same parity and that equation \ref{eq:top_same_parity} applies.

				Assume now that $\alpha_2 - 2\rho_2 = \alpha_0 - 2\rho_0$. Equations \ref{eq:lambda_elementary} then gives
				\[\frac{(\alpha_2+1)^2 - (\alpha_0+1)^2}{2} = 2\chi\]
				which is only possible if $\alpha_0$ and $\alpha_2$ have the same parity, and we conclude with equation \ref{eq:top_same_parity}.
			\end{proof}

			\noindent
			We refer to figures \ref{fig:01_138} and \ref{fig:01_297} for illustrations.

			The second group contains periods $\Lambda^k\Lambda_{0\ast}$ for any $\ast\in \{-, +\}$ and any integer $k\ge 0$.
			The case $k=0$ involves checking when $\Lambda_{0\ast}$ is elementary while
			the cases $k > 0$ require checking when $\Lambda$ and $\Lambda_{0\ast}$ are elementary together.
			\begin{Proposition}\label{prop:lambda0_elementary}
				For any $\ast\in \{-, +\}$, block $\Lambda_{0\ast} = (\chi', \chi, \chi) = ((\alpha_i', \rho_i'))_{0\le i\le 2}$ is elementary
				if and only if $\alpha_2' - 2\rho_2' = \alpha_0' - 2\rho_0'$.
			\end{Proposition}

			\begin{proof}
				Observe that we have, by definition,
				\begin{equation}\label{eq:lambda0_elementary}
					\begin{aligned}
						\frac{\alpha_0'(\alpha_0'+1)}{2} + \rho_0' + 1 &= \chi\ast 1\\
						\frac{\alpha_2'(\alpha_2'+1)}{2} + \rho_2' + 1 &= 3\chi\ast 1
					\end{aligned}
				\end{equation}
				The same reasoning as proposition \ref{prop:lambda_elementary} applies.
			\end{proof}

			\begin{Proposition}\label{prop:both_elementary}
				For any $\ast\in \{-, +\}$, blocks $\Lambda = ((\alpha_i, \rho_i))_{0\le i\le 2}$ and $\Lambda_{0\ast} = ((\alpha_i', \rho_i'))_{0\le i\le 2}$
				are elementary together if and only
				\begin{center}
					\begin{tabular}{ccc}
						$\alpha_2 - 2\rho_2 = \alpha_0 - 2\rho_0$ & and & $d_\ast(\alpha_0, \rho_0) = (\alpha_0, \rho_0\ast 1)$
					\end{tabular}
				\end{center}
			\end{Proposition}

			\begin{proof}
				If $\Lambda$ and $\Lambda_{0\ast}$ are elementary together, then from propositions \ref{prop:lambda_elementary}
				and \ref{prop:lambda0_elementary}, $\alpha_0$ and $\alpha_2$ have the same parity and
				$\alpha_0'$ and $\alpha_2'$ have the same parity. From proposition \ref{prop:elementary_rho2}, $\alpha_2' = \alpha_2$,
				and we deduce that $\alpha_0$ and $\alpha_0'$ have the same parity and that $\alpha_0 = \alpha_0'$, which
				is equivalent to $d_\ast(\alpha_0, \rho_0) = (\alpha_0, \rho_0\ast 1)$.

				For the other implication, simply observe that
				\begin{align*}
					&& \alpha_2 - 2\rho_2 &= \alpha_0 - 2\rho_0\\
					&\Leftrightarrow& \alpha_2 - 2(\rho_2\ast 1) &= \alpha_0 - 2(\rho_0\ast 1)\\
					&\Leftrightarrow& \alpha_2' - 2\rho_2' &= \alpha_0' - 2\rho_0'
				\end{align*}
				and conclude with propositions \ref{prop:lambda_elementary} and \ref{prop:lambda0_elementary}.
			\end{proof}

			\noindent
			A quick check will show that, for any $J\in \mathds{N}^*$,
			\begin{align*}
				\frac{3J}{3J\chi\ast 1} (3\chi) &= 3 - \frac{\ast 3}{3J\chi\ast 1} & \frac{3J}{3J\chi\ast 1} (3\chi\ast 1) &= 3 + \frac{\ast(3J - 3)}{3J\chi\ast 1}
			\end{align*}
			and with property \ref{property:all_ternary}, we deduce that, stackability allowing, $\Gamma_{J\ast}$
			would have elementary sequence $(\Lambda^{J-1}\Lambda_{0\ast})^\infty$ with initial ordinates
			\begin{align*}
				\frac{2}{3J\chi - 1} &< T_y(0) < \frac{3}{3J\chi - 1} & \frac{3J - 3}{3J\chi + 1} &< T_y(0) < \frac{3J - 2}{3J\chi + 1}
			\end{align*}
			when $\ast$ is $-$ and $+$ respectively.

			Figure \ref{fig:06_827} is given as illustration. The ball comes out of the green part moving northwest,
			meaning one should compare the green part with the reflection of the blue part along a vertical axis.
			\begin{figure}[h!]
				\begin{center}
					\scalebox{0.6}{\includegraphics{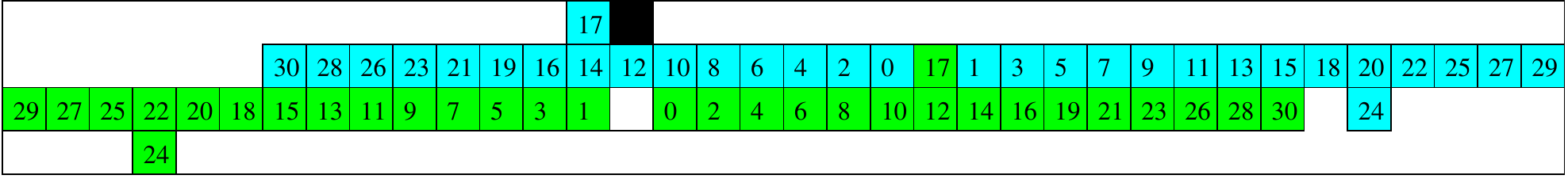}}
				\end{center}
				\vspace{-4mm}
				\caption{One relative period for slope $\frac{3\times 2}{3\times 2\times 138 - 1}$.}
				\label{fig:06_827}
			\end{figure}

			The third and last group of possible elementary sequence's periods contains $\Lambda_{1\ast}$ and $\Lambda_{2\ast}$.
			\begin{Proposition}\label{prop:lambda1_lambda2_elementary}
				For any $\ast\in \{-, +\}$, block $\Lambda_{1\ast}$ is elementary if and only if $\Lambda_{2\ast}$ is elementary. Additionally,
				with $\Lambda = ((\alpha_i, \rho_i))_{0\le i\le 2}$, blocks $\Lambda_{1\ast}$ and $\Lambda_{2\ast}$ are elementary
				if and only if $\alpha_2 - 2(\rho_2\ast 1) = -\alpha_0 + 2\rho_0 + 1$.
			\end{Proposition}

			\begin{proof}
				Observe that $(\alpha_0, \rho_0)$ and $d_\ast(\alpha_2, \rho_2)$ are the $0$th and $2$th coordinates
				of both $\Lambda_{1\ast}$ and $\Lambda_{2\ast}$. This takes care of the first claim.

				Assume $\Lambda_{2\ast}$ is elementary and that $\alpha_0$ and $\alpha_2$ have the same parity. Recall that
				$d_\ast(\alpha_2, \rho_2) = (\alpha_2, \rho_2\ast 1)$ from proposition \ref{prop:elementary_rho2}. Applying
				equation \ref{eq:top_same_parity} to $\Lambda_{2\ast}$ gives
				\[\alpha_2 - 2(\rho_2\ast 1) = \alpha_0 - 2\rho_0\]
				Using equations \ref{eq:lambda_elementary}, we obtain
				\[\frac{(\alpha_2+1)^2 - (\alpha_0+1)^2}{2} - (\ast 1) = 2\chi\]
				which is impossible as $\alpha_0$ and $\alpha_2$ having the same parity implies the left hand side is odd.
				We conclude that if $\Lambda_{2\ast}$ is elementary, then $\alpha_0$ and $\alpha_2$ must have different parities
				and that, again with proposition \ref{prop:elementary_rho2}, equation \ref{eq:top_different_parities} applies to $\Lambda_{2\ast}$.

				Assume now that $\alpha_2 - 2(\rho_2\ast 1) = -\alpha_0 + 2\rho_0 + 1$. Equations \ref{eq:lambda_elementary} then gives
				\[\frac{(\alpha_2)^2 - (\alpha_0+1)^2 + 2}{2} + 2\rho_2\ast 1 = 2\chi\]
				which is only possible if $\alpha_0$ and $\alpha_2$ have different parities, and we conclude with equation \ref{eq:top_different_parities}.
			\end{proof}

			We refer to figure \ref{fig:03_34} for an illustration.

			An interesting fact is that $\Lambda_{0\ast}, \Lambda_{1\ast}$ and $\Lambda_{2\ast}$ can be elementary together.
			Assuming they can also be self-stackable together, a single slope could then have three distinct elementary sequences
			$\Lambda_{0\ast}^\infty, \Lambda_{1\ast}^\infty$ and $\Lambda_{2\ast}^\infty$ (for different initial ordinates).
			The smallest, in terms of the leading slice, slope we could find is
			\begin{center}
				\begin{tabular}{cc}
					$\displaystyle{\frac{3}{3\times 232 - 1}}$ & with $\Lambda = ((21, 0), (29, 28), (36, 29))$\\
				\end{tabular}
			\end{center}
			for the $-$ case and
			\begin{center}
				\begin{tabular}{cc}
					$\displaystyle{\frac{3}{3\times 388046811731629680 + 1}}$ & with\\
				\end{tabular}
				\[\Lambda = ((880961759, 880961759), (1245868069, 163430944), (1525870528, 322454383))\]
			\end{center}
			for the $+$ case. The difference in the leading slices between the $-$ and $+$ cases is striking, but is actually
			a recurring issue (see section \ref{subsec:sufficient_conditions}). This, however, appears to be a common theme
			when dealing with quadratic diophantine equations.

			A quick note on how the $+$ case example was obtained. Contrary to the $-$ case, a program incrementally going through
			the values of the leading slice is, timewise, quite impractical. Setting $\rho_0 = \alpha_0$ along with the condition
			of proposition \ref{prop:lambda1_lambda2_elementary} and the equations of \ref{eq:lambda_elementary}, one can obtain
			the generalised Pell's equation given below:
			\begin{align*}
				(3\alpha_0+5)^2 - 3(\alpha_2+1)^2 = 1
			\end{align*}
			The (apparently unpublished) article \cite{Rob} provides algorithms and references to solve generalised Pell equations.
			In our case, we can refer to either \cite{Mol0} or \cite{Mol1}. From there, a program going exclusively through the solutions of
			this equation yielded a proper example almost instantly.

		\subsection{Stackability}\label{subsec:stackability}
			In the previous section, we have given conditions for the seven (separating the $-$ and $+$ cases) ternary blocks
			of lemma \ref{lemma:necessary_conditions} to be elementary. What is left is to examine the stackability properties
			of the eleven (again separating the $-$ and $+$ cases) pairs of blocks mentionned in lemma \ref{lemma:necessary_conditions}.
			While one could simply go through the eleven pairs of blocks and check the stackability conditions of section \ref{subsec:stacking},
			we give here a more convenient way to check all of them at once, by essentially checking a single one.\\

			With $X$ defined as in section \ref{subsubsec:ternary_blocks}, we define, for any ternary blocks
			$\Delta = ((\alpha_i, \rho_i))_{0\le i\le 2}$ and $\Delta' = ((\alpha_i', \rho_i'))_{0\le i\le 2}$,
			\begin{align}\label{eq:xbar}
				\overline{X}(\Delta, \Delta') &= X(|\Delta|) + (-1)^{\alpha_2+1} X(|\alpha_1', \rho_1'|) - \zeta(\Delta, \Delta')
			\end{align}
			where $\zeta(\Delta, \Delta') = 1$ if $\alpha_2$ and $\alpha_1'$ have the same parity and $0$ otherwise.
			From section \ref{subsec:stacking}, should $\Delta$ be elementary
			and $\Delta'$ stack on top of $\Delta$, then $\overline{X}(\Delta, \Delta')$ is the abscissa
			of the left edge of the bottom brick of $\Delta'$.

			We consider a ternary block $\Lambda = (\chi, \chi, \chi)$, for some $\chi\in \mathds{N}^*$ and define
			$\Lambda_{0\ast}, \Lambda_{1\ast}, \Lambda_{2\ast}$ as in section \ref{subsec:necessary_conditions}.
			Verifying the stackability of any of the eleven pairs of blocks mentioned in lemma \ref{lemma:necessary_conditions}
			requires applying $\overline{X}$ to each pair. Our strategy is to compute $\overline{X}(\Lambda, \Lambda)$ and derive the others.
			The abscissa difference between a given pair and the $(\Lambda, \Lambda)$ stacking is called an \textbf{offset}.
			With $\Delta$ and $\Delta'$ as above, the offset is computed as the sum $(\epsilon_0 + \epsilon_1)$ where
			\begin{center}
				\begin{tabular}{ll}
					$\epsilon_0 = \ast (-1)^{\alpha_2+1}$ & if $\Delta\in \{\Lambda_{0\ast}, \Lambda_{1\ast}, \Lambda_{2\ast}\}$,\\
					$\epsilon_0 = 0$ & otherwise\\\\
					$\epsilon_1 = \ast (-1)^{\alpha_2+\alpha_1}$ & if $\Delta'\in \{\Lambda_{0\ast}, \Lambda_{1\ast}, \Lambda_{2\ast}\}$ and $d_\ast(\alpha_1, \rho_1) = (\alpha_1, \rho_1\ast 1)$,\\
					$\epsilon_1 = 0$ & otherwise
				\end{tabular}
			\end{center}
			Note that when $\Delta'\in \{\Lambda_{0\ast}, \Lambda_{1\ast}, \Lambda_{2\ast}\}$ and $d_\ast(\alpha_1, \rho_1) \ne (\alpha_1, \rho_1\ast 1)$,
			the offsets created by the terms $(-1)^{\alpha_2+1} X(|\alpha_1', \rho_1'|)$ and $\zeta(\Delta, \Delta')$ of equation \ref{eq:xbar}
			cancel each other, leading to $\epsilon_1$ being null.

			We provide the following offset table.
			\vspace{-4mm}
			\begin{center}
				\resizebox{\hsize}{!}{
					\rowcolors{1}{}{gray!20}
					$\begin{array}{|c|c|c|c|c|c|c|c|c|c|}
						\hline
						&&& (\Lambda_{1-}, \Lambda_{1-}) & (\Lambda_{2-}, \Lambda_{2-}) &&&& (\Lambda_{2+}, \Lambda_{2+}) & (\Lambda_{1+}, \Lambda_{1+})\\
						\hline
						\alpha_2 & \alpha_1 & \rho_1 & (\Lambda_{0-}, \Lambda_{0-}) & (\Lambda_{0-}, \Lambda) & (\Lambda, \Lambda_{0-}) & (\Lambda, \Lambda) & (\Lambda, \Lambda_{0+}) & (\Lambda_{0+}, \Lambda) & (\Lambda_{0+}, \Lambda_{0+})\\
						\hline
						\text{even} & \text{even} & 0 < \rho_1 < \alpha_1 & 0 & 1 & -1 & 0 & 1 & -1 & 0\\
						\hline
						\text{-} & \text{-} & \rho_1 = 0 & 1 & 1 & 0 & 0 & 1 & -1 & 0\\
						\hline
						\text{-} & \text{-} & \rho_1 = \alpha_1 & 0 & 1 & -1 & 0 & 0 & -1 & -1\\
						\hline
						\text{even} & \text{odd} & 0 < \rho_1 < \alpha_1 & 2 & 1 & 1 & 0 & -1 & -1 & -2\\
						\hline
						\text{-} & \text{-} & \rho_1 = 0 & 1 & 1 & 0 & 0 & -1 & -1 & -2\\
						\hline
						\text{-} & \text{-} & \rho_1 = \alpha_1 & 2 & 1 & 1 & 0 & 0 & -1 & -1\\
						\hline
						\text{odd} & \text{even} & 0 < \rho_1 < \alpha_1 & 0 & -1 & 1 & 0 & -1 & 1 & 0\\
						\hline
						\text{-} & \text{-} & \rho_1 = 0 & -1 & -1 & 0 & 0 & -1 & 1 & 0\\
						\hline
						\text{-} & \text{-} & \rho_1 = \alpha_1 & 0 & -1 & 1 & 0 & 0 & 1 & 1\\
						\hline
						\text{odd} & \text{odd} & 0 < \rho_1 < \alpha_1 & -2 & -1 & -1 & 0 & 1 & 1 & 2\\
						\hline
						\text{-} & \text{-} & \rho_1 = 0 & -1 & -1 & 0 & 0 & 1 & 1 & 2\\
						\hline
						\text{-} & \text{-} & \rho_1 = \alpha_1 & -2 & -1 & -1 & 0 & 0 & 1 & 1\\
						\hline
					\end{array}$
				}
			\end{center}

			We conclude this section by observing that, from proposition \ref{prop:elementary_rho2}, the bounds given in section \ref{subsec:stacking}
			are always the same for all eleven pairs of blocks.

		\subsection{Sufficient conditions and mutual exclusions}\label{subsec:sufficient_conditions}
			Let $\Lambda = (\chi, \chi, \chi)$ with $\chi\in \mathds{N}^*$ be a ternary block and define
			$\Lambda_{0\ast}, \Lambda_{1\ast}, \Lambda_{2\ast}$ as in section \ref{subsec:necessary_conditions}.
			We say that slope $\frac{1}{\chi}$ satisfies condition
			\begin{itemize}
				\item $(\ceps)$ if $\Lambda$ is elementary and self-stackable,
				\item $(0\ast)$ (resp. $(1\ast), (2\ast)$) if $\Lambda_{0\ast}$ (resp. $\Lambda_{1\ast}, \Lambda_{2\ast}$) is elementary and self-stackable,
				\item $(\ceps0\ast)$ if $\Lambda$ and $\Lambda_{0\ast}$ are elementary, $\Lambda_{0\ast}$ stacks on top of $\Lambda$ and $\Lambda$
					stacks on top of $\Lambda_{0\ast}$.
			\end{itemize}

			With individual cases detailed in sections \ref{subsec:finding_elementary_blocks} and \ref{subsec:stackability},
			we can state the following general theorem.

			\begin{Theorem}\label{theorem:sufficient_conditions}
				Let $\Lambda = (\chi, \chi, \chi)$ with $\chi\in \mathds{N}^*$ be a ternary block and define
				$\Lambda_{0\ast}, \Lambda_{1\ast}, \Lambda_{2\ast}$ as in section \ref{subsec:necessary_conditions}.
				\begin{itemize}
					\item If $(\ceps)$ then slope $\frac{1}{\chi}$ is elementary with elementary sequence $\Lambda^\infty$.
					\item If $(0\ast)$ (resp. $(1\ast), (2\ast)$) then slope $\frac{3}{3\chi\ast 1}$
						is elementary with elementary sequence $\Lambda_{0\ast}^\infty$ (resp. $\Lambda_{1\ast}^\infty$, $\Lambda_{2\ast}^\infty$).
					\item If $(\ceps0\ast)$, then for any integer $J > 1$, slope $\frac{3J}{3J\chi\ast 1}$
						is elementary with elementary sequence $(\Lambda^{J-1}\Lambda_{0\ast})^\infty$.
				\end{itemize}
			\end{Theorem}

			In section \ref{subsec:finding_elementary_blocks}, we gave examples of slopes such that $(0\ast), (1\ast)$
			and $(2\ast)$ were true together. One may ask what combinations of the $9$ conditions
			(we separate the $-$ and $+$ cases) above are possible.

			\begin{Conjecture}\label{prop:condition_sets}
				The $19$ condition sets given in the table below are the only condition sets that can be satisfied.
			\end{Conjecture}

			\begin{center}
				\resizebox{\hsize}{!}{
					$\begin{array}{|c|c|c|}
						\hline
						\chi & \Lambda & \text{Condition set}\\
						\hline

						\rowcolor{gray!20}
						3 & ((1, 1), (2, 2), (3, 2)) & \{\}\\
						\hline

						11 & ((4, 0), (6, 0), (7, 4)) & \{2+\}\\
						\hline

						\rowcolor{gray!20}
						138 & ((16, 1), (22, 22), (28, 7)) & \{0-, \ceps0-, \ceps, \ceps0+, 0+\}\\
						\hline

						173 & ((18, 1), (25, 20), (31, 22)) & \{1+\}\\
						\hline

						\rowcolor{gray!20}
						190 & ((18, 18), (27, 1), (33, 8)) & \{2-, 1-\}\\
						\hline

						232 & ((21, 0), (29, 28), (36, 29)) & \{2-, 1-, 0-\}\\
						\hline

						\rowcolor{gray!20}
						370 & ((26, 18), (37, 36), (46, 28)) & \{\ceps0+, 0+\}\\
						\hline

						1828 & ((59, 57), (85, 0), (104, 23)) & \{1+, 2+\}\\
						\hline

						\rowcolor{gray!20}
						144992 & ((538, 0), (761, 42), (932, 197)) & \{\ceps, \ceps0+, 0+\}\\
						\hline

						280028 & ((747, 649), (1057, 902), (1295, 923)) & \{\ceps0-, \ceps, \ceps0+, 0+\}\\
						\hline

						\rowcolor{gray!20}
						475881 & ((975, 80), (1379, 251), (1689, 437)) & \{0-, \ceps0-\}\\
						\hline

						524625 & ((1023, 848), (1448, 173), (1773, 1223)) & \{0-, \ceps, 0+\}\\
						\hline

						\rowcolor{gray!20}
						1439478 & ((1696, 421), (2399, 155), (2938, 1042)) & \{0-, \ceps0-, \ceps\}\\
						\hline

						6529510 & ((3613, 818), (5110, 414), (6258, 4118)) & \{1-\}\\
						\hline

						\rowcolor{gray!20}
						163269038 & ((18069, 15622), (25554, 21840), (31298, 9062)) & \{2-\}\\
						\hline

						& ((84148175, 27772619), &\\
						& (119003491, 928753), &\\
						\multirow{-3}{*}{3540457747762020} & (145748915, 58572989)) & \multirow{-3}{*}{$\{0-, \ceps0-, \ceps, \ceps0+\}$}\\
						\hline

						\rowcolor{gray!20}
						& ((419996254, 337792260), &\\
						\rowcolor{gray!20}
						& (593964399, 532915491), &\\
						\rowcolor{gray!20}
						\multirow{-3}{*}{88198427234806646} & (727454852, 491521559)) & \multirow{-3}{*}{$\{0+\}$}\\
						\hline

						388046811731629680 & \makecell{((880961759, 880961759),\\(1245868069, 163430944),\\(1525870528, 322454383))} & \{0+, 1+, 2+\}\\
						\hline

						\rowcolor{gray!20}
						& ((5843883585, 544273409), &\\
						\rowcolor{gray!20}
						& (8264499423, 1448966453), &\\
						\rowcolor{gray!20}
						\multirow{-3}{*}{17075487680982441315} & (10121903283, 2683283258)) & \multirow{-3}{*}{$\{0-\}$}\\
						\hline
					\end{array}$
				}
			\end{center}
			\vspace{2mm}

			All remaining condition sets except for $\{1-, 0-\}$, $\{0+, 1+\}$ and $\{\ceps, 0+\}$
			could be proven to be impossible.

			In addition to being an example for the empty condition set, making it non-elementary itself, we note
			that slope $\frac{1}{3}$ also serve as an example of a slope that is primary for an elementary (not just ternary)
			block, $((1, 1), (2, 2), (3, 2))$, which is not self-stackable. Furthermore, it can be shown to produce relatively periodic orbits. In \cite{BF},
			relatively periodic orbits are obtained for $\frac{1}{3}$ for various initial configurations (the initial set of obstacles),
			none of which are specifically our setup (a single unit square without a brick).
			Figure \ref{fig:01_03} fills this gap. It depicts the preperiod necessary to reach relative periodicity and a single relative period
			divided into two centrally inverted (up to some block translations) parts.
			Initial ordinate can be any real number in $(0, \frac{1}{3})$.
			\begin{figure}[h!]
				\begin{center}
					\scalebox{0.6}{\includegraphics{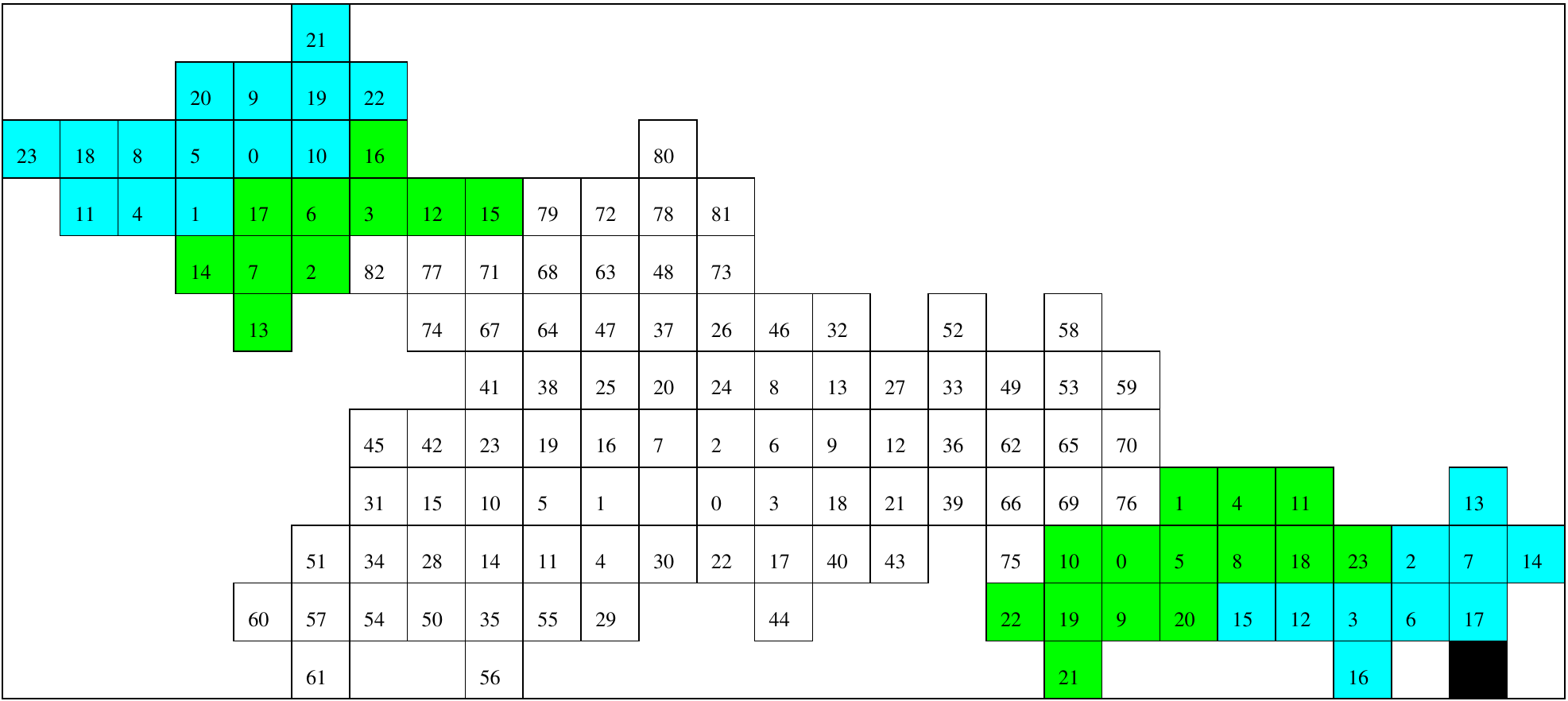}}
				\end{center}
				\vspace{-2mm}
				\caption{One relative period of $\frac{1}{3}$.}
				\label{fig:01_03}
			\end{figure}

			As a reminder, while elementary slopes have relatively periodic orbits, we have made no assumption on whether
			non-elementary slopes can produce relatively periodic orbits.

	\section{Symbolic orbits}\label{sec:symbolic_orbits}
		We define the symbolic orbit of a breakout orbit as the (right-)infinite word $\omega = (\omega_n)_{n\in \mathds{N}}$ of $\{a, b\}^\mathds{N}$
		such that $\omega_n = a$ if the $n$th brick was broken by hitting a horizontal edge and $\omega_n = b$ otherwise.\\

		It would have been a reasonable question to ask whether symbolic orbits could be used as substitutes
		for the geometrical ones, as is the case in billiards. We can now assert that this not the case.

		\begin{Property}\label{property:symbolic_orbit}
			If $\Lambda = (\chi, \chi, \chi) = ((\alpha_i, \rho_i))_{0\le i\le 2}$ and $\Lambda_{0\ast}$, for some $\ast\in \{-, +\}$,
			are both elementary blocks, then the (finite) symbolic word of a $\Lambda$-orbit is the same as the symbolic word of a $\Lambda_{0\ast}$-orbit
			if and only if we have $d_\ast(\alpha_i, \rho_i) = (\alpha_i, \rho_i\ast 1)$ for any $0\le i\le 2$.
		\end{Property}

		\begin{Proposition}
			A single symbolic orbit can be shared by an infinity of slopes.
		\end{Proposition}

		\begin{proof}
			Set $\chi = 138$ and $\Lambda = ((16, 1), (22, 22), (28, 7))$. From property \ref{property:symbolic_orbit},
			$\Lambda$-orbits and $\Lambda_{0-}$-orbits share the same symbolic word.
			Observe that $\frac{1}{\chi}$ satisfies condition $(\ceps0-)$, and apply theorem \ref{theorem:sufficient_conditions}.
		\end{proof}

		It should be noted however, that we can still construct an infinity of (pairwise) distinct symbolic orbits.

		\begin{Proposition}
			There exists an infinity of periodic symbolic orbits.
		\end{Proposition}

		\begin{proof}
			Set $\chi = 138$ and $\Lambda = ((16, 1), (22, 22), (28, 7))$, From property \ref{property:symbolic_orbit}
			$\Lambda$-orbits and $\Lambda_{0+}$-orbits do not share the same symbolic word.
			Observe that $\frac{1}{\chi}$ satisfies condition $(\ceps0+)$, and apply theorem \ref{theorem:sufficient_conditions}.
		\end{proof}

\end{document}